\documentclass[12pt]{article}
\usepackage{amssymb,amscd,amsthm,amsmath,color}
\usepackage{fullpage}
\newcommand{\PP}{\mathbb{P}}
\newcommand{\OO}{\mathcal{O}}

\newcommand{\MM}{\mathcal{M}}
\newcommand{\NN}{\mathcal{N}}
\newcommand{\FF}{\mathcal{F}}

\newcommand{\aA}{\mathcal{A}}
\newcommand{\BB}{\mathcal{B}}

\newcommand{\sing}{\textsf{sing}}
\newcommand{\smooth}{\textsf{smooth}}
\newcommand{\shrp}{\textsf{sharp}}

\newcommand{\cC}{\mathcal{C}}
\newcommand{\ev}{\operatorname{ev}}
\newcommand{\pr}{\operatorname{pr}}
\newcommand{\Spec}{\operatorname{Spec}}
\newcommand{\codim}{\operatorname{codim}}
\newcommand{\End}{\operatorname{End}}

\newtheorem{theorem}{Theorem}[section]
\newtheorem{lemma}[theorem]{Lemma}
\newtheorem{proposition}[theorem]{Proposition}
\newtheorem{corollary}[theorem]{Corollary}
\newtheorem{definition}[theorem]{Definition}
\newtheorem{conjecture}[theorem]{Conjecture}

\author{Eric Riedl and David Yang}

\title{Kontsevich spaces of rational curves on Fano hypersurfaces}
\begin{document}

\maketitle

\abstract{We investigate the spaces of rational curves on a general hypersurface.  In particular, we show that for a general degree $d$ hypersurface in $\PP^n$ with $n \geq d+2$, the space $\overline{\MM}_{0,0}(X,e)$ of degree $e$ Kontsevich stable maps from a rational curve to $X$ is an irreducible local complete intersection stack of dimension $e(n-d+1)+n-4$.  This resolves all but one case of a conjecture of Coskun, Harris and Starr, and also proves that the Gromov-Witten invariants of these hypersurfaces are enumerative.}

\section{Introduction}

A basic way of attempting to understand a variety is to understand its subvarieties.  For example, what subvarieties appear?  What are the irreducible components of their moduli spaces?  What are the dimensions of these components?  Much work has been done on this area, but the question appears to be extremely difficult in general, and there are many simple-sounding open questions.  For instance, the dimensions of the spaces of degree $d$, genus $g$ curves in $\PP^n$ are unknown in general.

When considering these questions, of particular interest is understanding the geometry of the rational subvarieties of a given variety, partly because these questions are sometimes more tractable and partly because the nice properties enjoyed by rational varieties give us particular insight into the geometry of the variety.  For instance, Harris, Mazur and Pandharipande \cite{HMP} prove that any smooth hypersurface in sufficiently high degree is unirational by looking at the space of $k$-planes contained in the hypersurface.  Or, in \cite{starrdompaper} Starr, again by considering the spaces of $k$-planes contained in a hypersurface, proves that the $k$-plane sections of a smooth hypersurface in sufficiently high dimension dominate the moduli space of hypersurfaces in $\PP^k$.  Rational varieties, particularly rational curves, have an important role in birational geometry and the Minimal Model Program, as is evidenced by the proofs by Campana \cite{Campana} and Koll\'ar-Miyaoka-Mori \cite{KMM} that Fano varieties are rationally connected.  In a different direction, genus $0$ Gromov-Witten invariants, which are important in mathematical physics, are an attempt to count the number of rational curves satisfying certain incidence conditions.  Knowing the dimensions of Kontsevich spaces allows us to connect the Gromov-Witten theory calculations to actual curves, proving that Gromov-Witten invariants are enumerative.

Here we focus on finding the dimensions of the spaces of rational subvarieties of hypersurfaces over an algebraically closed field of characteristic $0$.  The main result is that the spaces of rational curves on general hypersurfaces of degree $d \leq n-2$ in $\PP^n$ have the expected dimensions.

\section{Acknowledgements}
We would like to thank Joe Harris, Jason Starr, and Roya Beheshti for helpful conversations.  We would also like to thank the anonymous referee for many helpful comments.

\section{Previous work on Rational Curves}
\label{previousWork}
Let $X$ be a hypersurface in $\PP^n$ of degree $d$.  We are interested in the space of degree $e$ rational curves on $X$.  Dimensions of rational curves on Fano varieties have been studied for quite a while.  Mori \cite{Mori} proved that every point of a Fano manifold has a rational curve through it, and the spaces of rational curves through a general point of a Fano manifold were further studied by Koll\'ar-Miyaoka-Mori \cite{KMM}.  This sort of study naturally lead to questions about the dimension of the space of rational curves on a Fano manifold $X$.  For low degrees of curves there are many classical results.  See the textbooks of Koll\'{a}r \cite{kollar1996rational} and Debarre \cite{debarre} for more details on lines.  However, there were few results that applied for all degrees $e$ of rational curve, partly because of technical problems with the compact moduli spaces of rational curves in $\PP^n$ in use at the time: the Hilbert scheme and the Chow variety.

Some of these technical obstacles were removed when Kontsevich introduced the Kontsevich space $\overline{\MM}_{0,0}(\PP^n,e)$ of degree $e$ stable rational maps to $\PP^n$ and showed that it is smooth (as a stack) of the expected dimension \cite{kontsevich}.  Work of Kim-Pandharipande \cite{kimPandh} showed that these spaces were irreducible.  For more details on the Kontsevich space, including its construction, see \cite{FP}.  

In this paper, we consider $\overline{\MM}_{0,0}(X,e)$, that is, the subvariety of $\overline{\MM}_{0,0}(\PP^n,e)$ of rational maps whose image lies in $X$.  It is not hard to see that $\overline{\MM}_{0,0}(X,e)$ is cut out by a section of a rank $ed+1$ vector bundle on $\overline{\MM}_{0,0}(\PP^n,e)$.  Thus, every component of $\overline{\MM}_{0,0}(X,e)$ will have dimension at least $(e+1)(n+1)-4 - (ed+1) = e(n-d+1)+n-4$.  Let $R_e(X)$ be the union of all irreducible components of $\overline{\MM}_{0,0}(X,e)$ whose general elements correspond to generically injective maps from irreducible curves.  

There are examples of smooth hypersurfaces $X$ with $\dim R_e(X) > e(n-d+1)+n-4$.  For instance, it is well known that for any degree $d$, there exists a smooth hypersurface of degree $d$ containing linear spaces of half its dimension, i.e., there is a smooth hypersurface $Y \subset \PP^{n}$ with $n = 2k+1$ containing a linear space $\Lambda \cong \PP^k$.  In $\Lambda$, there will be $(e+1)(k+1)-4$ dimensional family of degree $e$ rational curves, which will be larger than the expected dimension $e(n-d+1)+n-4$ when $d > \frac{n+1}{2}+\frac{n-1}{2e}$.  Other than this example, we know of few other examples of Fano hypersurfaces containing too many rational curves, and it would be interesting to know of more examples, or to prove that for a certain degree range that no such examples exist.  Aside from the results of Beheshti \cite{beheshtidJDb} and the thesis of Pugin \cite{pugin}, little is known about the spaces of rational curves on arbitrary smooth hypersurfaces.

This leads naturally to the following open question, various special cases of which were conjectured by various people.  We state it as a Conjecture because it is a natural falsifiable statement which should form the framework for discussion for this problem, not necessarily because it has been conjectured in all cases.

\begin{conjecture}
\label{expDim}
For $X$ a very general hypersurface of degree $d$ in $\PP^n$ with $(n,d) \neq (3,4)$, the dimension of $R_e(X)$ is equal to $\max \{-1, e(n-d+1)+n-4 \}$, the minimum possible.
\end{conjecture}

As usual, we say a property holds for a ``very general'' hypersurface when it holds outside a countable union of proper closed subvarieties in the moduli space of hypersurfaces.  Conjecture \ref{expDim} is open for large ranges of $d$, $n$ and $e$.  However, some cases are known, and we summarize them here.  For $e \leq d+2$, the conjecture follows from a result of Gruson, Lazarsfeld and Peskine \cite{GLP}.  Furukawa makes this connection explicit and also proves a weaker result that holds in arbitrary characteristic \cite{Furukawa}.

For $n=3, d=4$, Conjecture \ref{expDim} is known to be false.  Very general quartic surfaces in $\PP^3$ all contain (nodal) rational curves, while the expected dimension of rational curves is $-1$ for all $e$.  This result was known to Mumford, although there has been considerable work on understanding exactly the degrees in which the curves appear, and the counts of these curves \cite{JunLi}.

In the special case $d=5$, $n=4$, Conjecture \ref{expDim} is a version of the well-known Clemens's Conjecture, which has been worked on by Katz, Kleiman-Johnsen, Cotterill and others \cite{Katz1986, Johnsen1995, Johnsen1997, Cotterill04rationalcurves, Cotterill07rationalcurves}.  Despite all of this progress, Clemens's Conjecture remains open for $e \geq 12$.

Some work has been done on Conjecture \ref{expDim} for $d$ larger than $n$.  Voisin \cite{voisin1996, voisin1998}, improving on work of Clemens \cite{clemensCurveHyp} and Ein \cite{Ein1988, Ein1991}, proved as a special case of a more general result that if $d \geq 2n-2$, then a general $X$ contains no rational curves.  Pacienza \cite{Pacienza2d-3} proves that for $d=2n-3$ and $n > 4$, $X$ contains lines but no other rational curves.

In the case $d<n$, Conjecture \ref{expDim} is a special case of a more general Conjecture of Coskun-Harris-Starr.

\begin{conjecture}
\label{flatConj}
For $X$ a general degree $d$ hypersurface in $\PP^n$ with $n \geq d+1$, then $\overline{\MM}_{0,0}(X,e)$ has the expected dimension $e(n-d+1)+n-4$, and the evaluation morphism $\ev : \overline{\MM}_{0,1}(X,e) \to X$ is flat of relative dimension $e(n-d+1)-2$.
\end{conjecture}

Clearly, Conjecture \ref{flatConj} for a given $n$, $d$, and $e$ implies Conjecture \ref{expDim} for that $n$, $d$ and $e$.  Notice that Conjecture \ref{flatConj} is for a general hypersurface, as opposed to a very general hypersurface.  Since there are countably many degrees of curve, it seems more natural to make conjectures of this form for very general hypersurfaces.  However, Harris-Roth-Starr \cite{HRS} showed that knowing Conjecture \ref{flatConj} for $e$ up to a certain threshold degree $\frac{n+1}{n-d+1}$ proves Conjecture \ref{flatConj} for all $e$, and used this to prove Conjecture \ref{flatConj} for  $d \leq \frac{n+1}{2}$.  In particular, when using their technique it suffices to consider general hypersurfaces, not very general ones.  Beheshti and Kumar \cite{BeheshtiKumar} prove Conjecture \ref{flatConj} for $d \leq \frac{2n+2}{3}$.

In this paper, we improve on Beheshti-Kumar's \cite{BeheshtiKumar} and Harris-Roth-Starr's \cite{HRS} results.  \footnote{As we were working on this write-up, we received word that Roya Beheshti had independently proven Conjecture \ref{expDim} for $d < n - 2\sqrt{n}$.  Her techniques seem likely to apply to hypersurfaces in the Grassmannian.}

\begin{theorem}
\label{firstTheoremStatement}
Let $X$ be a general degree $d$ hypersurface in $\PP^n$, with $n \geq d+2$.  Then $\overline{\MM}_{0,0}(X,e)$ is an irreducible local complete intersection stack of the expected dimension $e(n-d+1)+n-4$ and $\ev: \overline{\MM}_{0,1}(X,e) \to X$ is flat with fibers of dimension $e(n-d+1)-2$.
\end{theorem}

The idea of the proof in Harris-Roth-Starr \cite{HRS} is to prove Conjecture \ref{flatConj} for $e=1$, then to show that every rational curve through a given point specializes to a reducible curve.  By flatness of the evaluation morphism, the result follows by induction.  Beheshti-Kumar \cite{BeheshtiKumar} get a stronger result by proving flatness of the evaluation morphism for $e=2$ and then using the Harris-Roth-Starr result.  The key to Harris-Roth-Starr's approach is a version of Bend-and-Break which allows them to show that when there are enough curves in the fibers of the evaluation morphism $\ev : \overline{\MM}_{0,1}(X,e) \to X$, every curve must specialize to a reducible curve.  The reason that their bound applies only for small $d$ is because when $d$ gets larger, there are not enough curves to ensure that every component of a fiber of $\ev$ contains reducible curves.

Our approach builds on Harris-Roth-Starr's by borrowing curves from nearby hypersurfaces to ensure that there are enough curves to apply Bend-and-Break.  We do this by bounding the codimension of the space of hypersurfaces for which the statement of Conjecture \ref{expDim} does not hold.  Crucial to our analysis are the notions of \emph{$e$-level} and \emph{$e$-layered} hypersurfaces, that is, hypersurfaces that have close to the right dimensions of degree-at-most-$e$ rational curves through any given point and hypersurfaces whose rational curves all specialize to reducible curves.  Our proof proceeds by inductively bounding the codimension of the space of hypersurfaces that are not $e$-layered.

The rest of this document is organized as follows.  First we state and prove the version of Bend-and-Break that we will use.  Then we sketch how the version of Bend-and-Break can be used to prove the $d \leq \frac{n+1}{2}$ result of Harris-Roth-Starr.  Next, we introduce the concepts of $e$-levelness and $e$-layeredness and prove some important properties of them.  Then we prove Theorem \ref{firstTheoremStatement} using these notions.

\section{Background for Rational Curves on Hypersurfaces}

We will use standard facts about Kontsevich spaces, such as those found in \cite{FP}.  We treat $\overline{\MM}_{0,0}(X,e)$ as a coarse moduli space.  Occasionally we will need to use the result found in Vistoli \cite{vistoli_intersection} which says that there is a scheme $Z$ that is a finite cover of the stack $\overline{\MM}_{0,0}(X,e)$.  We also need the following well-known result.

\begin{lemma}
\label{kdimConds}
Let $Z$ be a $k$-dimensional variety.  Then the space of degree $d$ hypersurfaces containing $Z$ is codimension at least $\binom{d+k}{k}$ in the space of all hypersurfaces.
\end{lemma}
\begin{proof}
Using automorphisms of $\PP^n$, we can degenerate $Z$ to a (possibly non-reduced) scheme supported on a $k$-plane.  The space of degree $d$ hypersurfaces containing the $k$-plane has dimension $N-\binom{d+k}{k}$, where $N=\binom{n+d}{d}-1$, and so the dimension of the space of hypersurfaces containing the degeneration of $Z$ will be at most $N-\binom{d+k}{k}$, and hence, the dimension of the space of hypersurfaces containing $Z$ is at most $N-\binom{d+k}{k}$.
\end{proof}

The following variant of a result whose proof we read in \cite{BeheshtiKumar} (although it was known before this) is the version of Bend-and-Break that we will use.

\begin{proposition}
\label{twoSections}
For $T \subset \overline{\MM}_{0,0}(\PP^n,e)$ a complete subvariety with $\dim T = 1$, suppose each of the maps parameterized by $T$ contains two distinct fixed points $p, q \in \PP^n$ in its image.  Then $T$ parameterizes maps with reducible domains. 
\end{proposition}
\begin{proof}
This follows from the proof of Proposition 3.2 of Debarre \cite{debarre}.

Because the result is so important to what follows, we offer a simple proof in the case where the rational curves sweep out a surface (which is all we need for our application).  Suppose the result is false.  Then, possibly after a finite base-change, we can find a family of Kontsevich stable maps parameterized by $T$ providing a counterexample.  After normalizing, we can assume $T$ is smooth.  Thus, we have a $\PP^1$-bundle $\pi:B \to T$ and a map $\phi: B \to \PP^n$ such that the restriction of $\phi$ to each fiber of $\pi$ is the Kontsevich stable map in question.

The Neron-Severi group of $B$ is two dimensional, generated by the fiber class and a divisor whose class is $\OO(1)$ on each fiber.  Since the image of $B$ is two-dimensional, any contracted curves must have negative self-intersection.  Thus, the sections of $\pi$ corresponding to the two points $p$ and $q$ must be two disjoint curves with negative self-intersection.  Thus, their classes in the Neron-Severi group must be independent, since their intersection matrix is negative definite.  However, this contradicts the Neron-Severi rank being two, since it is impossible for the entire Neron-Severi group to be contracted.
\end{proof}

Note that Proposition \ref{twoSections} can be seen to be sharp by taking the image of $\overline{\MM}_{0,0}(\PP^n,1)$ in $\overline{\MM}_{0,0}(\PP^n,e)$ under the map induced by a sufficiently general self-map of $\PP^n$ where $\OO_{\PP^n}(1)$ pulls back to $\OO_{\PP^n}(e)$.

\begin{corollary}
\label{bendBreak}
If $T$ is a complete family in $\overline{\MM}_{0,0}(\PP^n,e)$ of dimension at least $2n-1$, then $T$ contains elements with reducible domains.
\end{corollary}
\begin{proof}
Consider the incidence correspondence $Y = \{ (C,f,p,q) \: | \: (C,f) \in T \: \textrm{and} \: p,q \in f(C) \}$.  Then $Y$ has dimension at least $2n+1$.  Looking at the natural map $Y \to \PP^n \times \PP^n$, we see that the general non-empty fiber has to be at least one-dimensional.  Thus, we can find a $1$-parameter subfamily passing through two distinct points.
\end{proof}

We also need a similar result for families of curves lying on a hypersurface all passing through one point:

\begin{proposition}
\label{oneSection}
Let $X$ be a hypersurface in $\PP^n$. If $T$ is a complete family in $\overline{\MM}_{0,0}(X,e)$ of dimension at least $n-1$ such that the image of each curve contains a fixed point $p$, then $T$ parameterizes a map with reducible domain.
\end{proposition}
\begin{proof}
Consider the incidence correspondence $Y = \{ (C,f,q) \: | \: (C,f) \in T \: \textrm{ and } \: q \in f(C) \}$.  Then $Y$ has dimension at least $n$.  Looking at the natural map $Y \to X$, we see that the general fiber has to be at least one-dimensional.  Thus, we can find a $1$-parameter subfamily passing through $p$ and another point $q$.
\end{proof}

\section{The case of small degree}
\label{HRSsummary}
Because the ideas in Harris-Roth-Starr \cite{HRS} are so central to our approach, we provide a sketch of the proof of a main result from Harris-Roth-Starr.

\begin{theorem}[Harris-Roth-Starr]
Suppose $d \leq \frac{n+1}{2}$.  Then if $X$ is a general degree $d$ hypersurface in $\PP^n$, the space of degree $e$ rational curves in $X$ through an arbitrary point $p$ has dimension $e(n-d+1)-2$.
\end{theorem}

\begin{proof}
(sketch) It follows from Proposition \ref{1levelness} that a general hypersurface has a $1 \cdot (n-d+1) - 2 = n-d-1$ dimensional family of lines through every point.  Now we use induction.  Suppose we know the result for all curves of degree smaller than $e$.  Then the space of reducible curves through a point $p$ with components of degrees $e_1$ and $e_2$ has dimension $e_1(n-d+1)-2 + 1 + e_2(n-d+1) - 2 = e(n-d+1)-3$.  Thus, it remains to show that every component of rational degree $e$ curves through $p$ contains reducible curves, since we know that the space of reducible curves is codimension at most $1$ in the space of all rational curves (see \cite{FP} for more information about the boundary divisor in Kontsevich space).

It follows from Proposition \ref{oneSection} that any $(n-1)$-dimensional family of curves contained in $X$ passing through $p$ must contain reducible curves.  Thus, we will have the result if
\[ e(n-d+1)-2 \geq n-1 \]
for all $e \geq 2$.  This simplifies to
\[ ed \leq n(e-1) + e - 1 \]
or
\[ d \leq \frac{(n+1)(e-1)}{e} .\]

The right-hand side is increasing in $e$, so if $d \leq \frac{n+1}{2}$ we have our result.
\end{proof}

From the proof, we see that if $\ev: \overline{\MM}_{0,1}(X,k) \to X$ is flat with expected-dimensional fibers for $1 \leq k \leq e-1$ but not for $k=e$, then $e(n-d+1)-2 \leq n-1$, or $e \leq \frac{n+1}{n-d+1}$.  That is, we need only check flatness for degrees up to $\frac{n+1}{n-d+1}$.  Harris-Roth-Starr call $\lfloor \frac{n+1}{n-d+1} \rfloor$ the threshold degree.  Note that as a Corollary of the proof, it follows that every component of degree $e$ curves contains curves with reducible domains.

Harris-Roth-Starr also prove irreducibility of the space of rational curves, and we will need this result as well, but we will describe it further in the next section, after we have talked about $e$-layeredness.

\section{$e$-levelness and $e$-layeredness}
This section is about two related concepts that underlie the ideas behind our proofs: $e$-levelness and $e$-layeredness.  Roughly speaking, an $e$-level point of a hypersurface has the expected dimension of degree-up-to-$e$ rational curves through it, and an $e$-layered point is such that every degree-up-to-$e$ rational curve through it specializes to a reducible curve.  The definitions are new, but they are related to ideas in \cite{HRS}.  Our main innovation is extending these ideas to singular hypersurfaces, so that we can try to bound the codimensions of the loci of hypersurfaces which are not $e$-layered.

\begin{definition}
A point $p \in X$ is \emph{$e$-level} if:
\begin{itemize}
\item $p \in X_{\smooth}$ and the space of rational curves in $X$ through $p$ has dimension at most $e(n-d+1)-2$ or
\item $p \in X_{\sing}$ and the space of rational curves in $X$ through $p$ has dimension at most $e(n-d+1)-1$.
\end{itemize}
A point $p \in X$ is \emph{$e$-sharp} if it is not $e$-level.
\end{definition}

The space $T$ of degree $e$ rational curves in $X$ through $p$ has the \emph{expected dimension} if $p$ is singular and $\dim T = e(n-d+1)-1$ or $p$ is smooth and $\dim T = e(n-d+1)-2$.  The reason that the condition is different for singular points is that through a singular point, there will always be at least a $(e(n-d+1)-1)$-dimensional family of rational curves, as we can see from writing out how many conditions it is for an explicit map from $\PP^1$ to $\PP^n$ to lie in $X$.  Points are $e$-level if they have the expected dimension of degree $e$ rational curves through them.

\begin{definition}
A hypersurface $X$ is \emph{$e$-level} if for every $k \leq e$ the following two conditions hold:
\begin{itemize}
\item There are no rational curves of degree $k$ contained in $X_{\sing}$.
\item Every point of $X$ is $k$-level.
\end{itemize}
A hypersurface is \emph{$e$-sharp} if it is not $e$-level.
\end{definition}

Define
\[ \Phi = \{ (p,X) | p \in X \}  \subset \PP^n \times \PP^N .\]
Let $\Phi_{\smooth}$ and $\Phi_{\sing}$ be the respectively open and closed subsets given by
\[ \Phi_{\smooth} = \{ (p,X) | p \in X \textrm{ such that $p \in X_{\smooth}$} \}  \subset \Phi \]
and
\[ \Phi_{\sing} = \{ (p,X) | p \in X \textrm{ such that $p \in X_{\sing}$} \}  \subset \Phi .\]

Let $\Phi_{e,\shrp} \subset \Phi$ be the locus of pairs $(p,X)$ where $p$ is an $e$-sharp point of $X$.  Notice that $\Phi_{e,\shrp}$ is not closed in $\Phi$.  To see this, consider the family of cubics in $\PP^5$ cut out by $f_t=tx_0^2x_1+x_0x_1x_2+x_1^3+x_2^3+x_3^3+x_4^3+x_5^3$.  For all $t$, there is a $2$-dimensional family of lines through the point $p = [1:0:0:0:0:0]$.  For $t \neq 0$, $X_t = V(f_t)$ is smooth $p$, which means that $p$ is a $1$-sharp point of $X_t$.  However, for $t = 0$, $X_0$ is singular at $p$, which means $p$ is a $1$-level point of $X_0$.  Although $\Phi_{e,\shrp}$ is not closed in $\Phi$, it is the case that $\Phi_{e,\shrp} \cap \Phi_{\smooth}$ is closed in $\Phi_{\smooth}$ and $\Phi_{e,\shrp} \cap \Phi_{\sing}$ is closed in $\Phi_{\sing}$ (which means that it is also closed in $\Phi$, since $\Phi_{\sing}$ is closed in $\Phi$).

\begin{definition}
A point $p \in X$ is \emph{$e$-layered} if:
\begin{itemize}
\item It is $1$-level.
\item For every $k \leq e$, every irreducible component parameterizing degree $k$ rational curves through $p$ contains reducibles.
\end{itemize}
A point $p \in X$ is \emph{$e$-uneven} if it is not $e$-layered.
\end{definition}

As before, we define $\Phi_{e,\text{layered}}$ to be $\{ (p,X) | \: p \text{ is an $e$-layered point of $X$} \}$ and $\Phi_{e,\text{uneven}}$ to be $\{ (p,X) | \: p \text{ is an $e$-uneven point of $X$} \}$.  The definition of $e$-layered hypersurfaces is analogous to that of $e$-level hypersurfaces.

\begin{definition}
A hypersurface $X$ is \emph{$e$-layered} if:
\begin{itemize}
\item All of its points are $e$-layered.
\item It contains no rational curves of degree less than or equal to $e$ in its singular locus.
\end{itemize}
\end{definition}

Proposition \ref{expRedDim} will allow us to relate layeredness and levelness.  We wish to specialize arbitrary rational curves to reducible curves in order to get a bound on the dimension of the space of curves through a point.  The specialization will not be useful unless we know that the space of reducible curves has sufficiently small dimension.  The notion of $e$-levelness is exactly tailored so that this is the case.

\begin{proposition}
\label{expRedDim}
Let $X$ be an $(e-1)$-level degree $d$ hypersurface.  Denote by $\ev: \overline{\MM}_{0,1}(X,e) \to X$ the evaluation morphism.  Then if $p$ is a point of $X$, the subspace of reducible curves in $\operatorname{ev}^{-1}(p)$ has dimension at most $e(n-d+1)-3$ if $p \in X_{\smooth}$ and at most $e(n-d+1)-2$ if $p \in X_{\sing}$.
\end{proposition}
\begin{proof}
We use strong induction on $e.$ It is obvious for $e=1$, as there are no reducible curves in $\overline{\mathcal{M}}_{0,1}(X,1).$

Denote by $\overline{\mathcal{M}}_{0,\{a\}}(X,e)$ the Kontsevich space of stable degree $e$ maps from rational curves with a marked point $a$.  The subspace of reducible curves in $\overline{\mathcal{M}}_{0,\{a\}}(X,e)$ is covered by maps from $\BB_{e_1} = \overline{\mathcal{M}}_{0,\{a,b\}}(X,e_1)\times_X \overline{\mathcal{M}}_{0,\{c\}}(X,e_2)$, where $e_1+e_2=e$, $e_1, e_2 \geq 1$, and the maps $\operatorname{ev}_b: \overline{\mathcal{M}}_{0,\{a,b\}}(X, e_1) \to X$ and $\operatorname{ev}_c: \overline{\mathcal{M}}_{0,\{c\}}(X,e_2) \to X$ are used to define the fiber products. The map from $\overline{\mathcal{M}}_{0,\{a,b\}}(X,e_1)\times_X\overline{\mathcal{M}}_{0,\{c\}}(X,e_2)$ to $\overline{\mathcal{M}}_{0,\{a\}}(X,e)$ is defined by gluing the domain curves together along $b$ and $c$ (see \cite{behrendManin} for details on how the gluing map works). The marked point $a$ of the first curve becomes the point $a$ of the resulting curve.  Let $\pr_1$ and $pr_2$ be the projection maps of $\BB_{e_1}$ onto the first and second components.

Fix a point $p \in X$, and write $Y = (\ev_a \circ \pr_1)^{-1}(p) \subset \BB_{e_1}$.  We are thus reduced to bounding the dimension of $Y$.  Define $Z = \ev_a^{-1}(p) \subset \overline{\MM}_{0,\{a\}}(X,e_1)$, and $Z' = \ev_a^{-1}(p) \subset \overline{\MM}_{0,\{a,b\}}(X,e_1)$, so that we have a natural sequence of maps $Y \to Z' \to Z$ given by $\pr_1: Y \to Z'$ and the forgetful map $\pi: Z' \to Z$ which forgets the point $b$.  

\begin{center}
$\begin{CD}
Y       @>>>  \BB_{e_1} @>>> \overline{\MM}_{0,\{a\}}(X,e)\\
@VVV          @VV\pr_1V\\
Z'      @>>>  \overline{\MM}_{0,\{a,b\}}(X,e_1)\\
@VVV          @VV\pi V\\
Z       @>>>  \overline{\MM}_{0,\{a\}}(X,e_1)\\
@VVV          @VV\ev_a V\\
\Spec{k}@>>p> X
\end{CD}$
\end{center}

Given a tuple $(f,C,p_a) \in Z \subset \overline{\MM}_{0,\{a\}}(X,e_1)$, we wish to analyze the fibers of $\pr_1 \circ \pi$.  The fibers of $\pi$ are all $1$-dimensional, and the fibers of $\pr_1$ are all at most $e_2(n-d+1)-1$ dimensional by $(e-1)$-levelness.  If $C$ is irreducible, then since there are no degree at most $e-1$ rational curves in $X_{\sing}$, the general fiber of $\pr_1$ over a point $(f,C,p_a,p_b) \in Z'$ has dimension $e_2(n-d+1)-2$.  By induction, we know that for a general $(f,C,p_a) \in Z$ with $f(p_a) = p$, $C$ is irreducible.  Therefore, $\dim Y = \dim Z + 1 + e_2(n-d+1) - 2 = \dim Z + e_2(n-d+1)-1$.

Putting it all together, by $e$-levelness, the dimension of $Z$ is at most $e_1(n-d+1)-2$ if $p \in X_{\smooth}$, and $e_1(n-d+1)-1$ if $p \in X_{\sing}$.  Thus, the dimension of $Y$ is $e_1(n-d+1) - 2 + e_2(n-d-1) - 1 = e(n-d+1)-3$ if $p \in X_{\smooth}$ or at most $e_1(n-d+1)-1 + e_2(n-d+1) - 1 = e(n-d+1)-2$ if $p \in X_{\sing}$, as desired.

\end{proof}

Since reducible curves are codimension at most one in $\overline{\MM}_{0,1}(X,e)$, we obtain a corollary.

\begin{corollary}
If $X$ is $e$-layered, then $X$ is $e$-level.
\end{corollary}
\begin{proof}
The result follows from a simple induction argument.  We see that $X$ is $1$-level by definition.  Now suppose $X$ is $(e-1)$-level.  Then the space of reducibles through an arbitrary point of $X$ will have dimension at most $e(n-d+1)-3$ if $p \in X_{\smooth}$ and $e(n-d+1)-2$ if $p \in X_{\sing}$.  Since $X$ is $e$-layered, every component of degree $e$ curves through $p$ will contain reducibles, which means that it will have dimension at most $e(n-d+1)-2$ if $p \in X_{\smooth}$ or $e(n-d+1)-1$ if $p \in X_{\sing}$.  Thus, $X$ is $e$-level.
\end{proof}

It follows that if $p \in X$ is $(e-1)$-level, but $e$-sharp, than $p$ must be $e$-uneven.  

\begin{corollary}
\label{closedUneven}
Let $(X_t,p_t,B_t)$ be an irreducible family of triples where $p_t$ is a point of the $(e-1)$-level hypersurface $X_t$ and $B_t$ is a component of the family of curves in $X_t$ through $p_t$.  For $t \neq 0$, suppose $B_t$ contains no reducibles, (which implies that $p_t$ is an $e$-uneven point of $X_t$).  Then $B_0$ contains no reducible curves.  In particular, the space $\Phi_{e,\text{uneven}} \cap \Phi_{e-1,\text{level}}$ is closed in $\Phi_{e-1,\text{level}}$.
\end{corollary}
\begin{proof}
To get a contradiction, suppose we have a $1$-parameter family of triples $(p_t,X_t,B_t)$ where $p_t$ is an $(e-1)$-level point of $X_t$, such that $B_t$ contains no reducibles for $t \neq 0$ but $B_0$ does contain reducibles.  Since $(p_0,X_0)$ is $(e-1)$-level, this means that the family of reducible curves in $X_t$ passing through $p_t$ is codimension at least $2$ in the family of all degree curves in $X_t$ passing through $p_t$.  This is a contradiction.
\end{proof}

Using Proposition \ref{expRedDim}, we immediately obtain a generalization of the result of Harris-Roth-Starr.  The idea behind this corollary is related to ideas found in Harris, Roth, and Starr's treatment of families of curves on smooth cubic hypersurfaces \cite{HRS2}.

\begin{corollary}
\label{lotsaCurvesCor}
For every triple $(n,d,e)$ with $n \geq d+2$ and $(e+1)(n-d+1)-2 \geq n-1$, for every $e$-level, degree $d$ hypersurface $X$ in $\PP^n$ which has no rational curves of any degree in its singular locus, then $X$ is $k$-level for all $k$.  It follows that if $X$ contains no rational curves in its singular locus and is $\lfloor \frac{n+1}{n-d+1} \rfloor$-level, then $X$ is $k$-level for all $k$.
\end{corollary}
As in Section \ref{HRSsummary}, we refer to $\lfloor \frac{n+1}{n-d+1} \rfloor$ as the threshold degree.
\begin{proof}
Because for $k \geq e+1$ every irreducible component of the space of degree $k$ rational curves through a point must have dimension at least $(e+1)(n-d+1)-2 \geq n-1$, we see by Proposition \ref{oneSection} that every component of the space of degree $k$ rational curves through an arbitrary point of $X$ must contain reducible curves.  Since the space of reducible curves is a divisor in $\overline{\MM}_{0,0}(\PP^n,e)$, by Proposition \ref{expRedDim}, we see that every component of the space of degree $k$ rational curves through an arbitrary point must have the expected dimension.
\end{proof}

The following result is essentially proven in \cite{HRS}, although they do not have the term $e$-layered.  For convenience, we sketch a slightly modified version of their proof here (the proof in \cite{HRS} does not specialize all the way to brooms but instead argues strictly using trees of lines, but the main ideas are the same).

\begin{theorem}
If $n \geq d+2$ and $X \subset \PP^n$ is a smooth, degree $d$, $e$-layered hypersurface such that the space of lines through a general point is irreducible, then $\overline{\MM}_{0,1}(X,e)$ is irreducible of the expected dimension.
\end{theorem}
\begin{proof}
By induction, it follows that the space of degree $e$ rational curves through a general point $p$ contains curves that are trees of lines with no nodes at $p$.  By $1$-levelness, it follows that any tree of lines can be specialized to a ``broom'' of lines, that is, a set of $e$ lines all passing through the same point which is distinct from the fixed point $p$.  By irreducibility of the space of lines through a general point, it follows that the space of brooms is irreducible, and considering codimensions shows that every component of the space of rational curves through $p$ that contains a broom contains the entire space of brooms.

For a general broom, the lines that it contains will all have balanced normal bundle in $X$.  (Recal that a vector bundle $E$ on $\PP^1$ is \emph{balanced} if $H^1(\End(E)) = 0$).  It follows that the pullback of the tangent sheaf $T_X$ of $X$ twisted by $-p$ will have no $H^1$ for a general element of the family of brooms.  Thus, the family of brooms is contained in a unique irreducible component of $\overline{\MM}_{0,1}(X,e)$.  However, since we showed that it is contained in every component of $\overline{\MM}_{0,1}(X,e)$, our result is proven.
\end{proof}

Let $S_1$ be the closure of the set of $1$-sharp hypersurfaces, and let $S_e \subset \PP^N$ be the closure of the union of $S_1$ with the set of $e$-uneven hypersurfaces.  Note that $S_k \subset S_e$ for $k \leq e$.  

The way we prove Theorem \ref{mainResult} is by bounding the codimension of the locus of $e$-uneven hypersurfaces.  We prove the base case $e=1$ here.  The ideas are similar to those found in Section 2 of \cite{HRS}, but we need more precise dimension estimates so we restate and re-prove the result.

\begin{proposition}
\label{1levelness}
If $n \geq d+2$, then the codimension of $S_1$ in $\PP^N$ is at least $\min \{ n(d-2)+3,\binom{n}{2} - n + 1 \}$.
\end{proposition}
\begin{proof}
Note that $S_1$ will simply be the space of hypersurfaces singular along a line union the closure of the space of hypersurfaces having a sharp point.

First consider the space of hypersurfaces everywhere singular along a given line.  Let $f$ be the polynomial cutting out our hypersurface.  We examine what conditions are imposed on the coefficients of $f$ when we insist that $V(f)$ be everywhere singular along a given line.  If we choose coordinates so that the line is $x_1 = \cdots = x_{n-1} = 0$, then the coefficients of $x_0^{j}x_n^{d-j}$ will have to vanish for $0 \leq j \leq d$, as will the coefficients of $x_ix_0^{j}x_n^{d-j-1}$ for $0 \leq j \leq d-1$.  This is $nd+1$ conditions.  Since there is a $(2n-2)$-dimensional family of lines in $\PP^n$, this means that for a hypersurface to be singular along a line is $nd+1 - (2n-2) = n(d-2)+3$ conditions.

Thus, it remains to bound the codimension of the space of hypersurfaces with a sharp point.  By considering the natural projection map $\Phi \to \PP^n$, it will suffice to show that the codimension of $\Phi_{1,\shrp} \subset \Phi$ is at least $\binom{n}{2}$.

We do this by considering the fibers of the map $\phi: \Phi \to \PP^n$.  Because all points of $\PP^n$ are projectively equivalent, it will suffice to work with a single fiber of $\phi$, say the fiber over $p$.  Choose homogeneous coordinates with $p = [1,0,\cdots, 0]$, which means that in the affine patch $D_+(x_0) = \mathbb{A}^n$ we have $p = (0,\cdots,0)$, and let $f$ be an equation cutting out a hypersurface $X$ which contains $p$.  We want to understand how many conditions are imposed on the coefficients of $f$ when we insist that $p$ be a $1$-sharp point of $X$.  Take the Taylor expansion of $f$ at $p$ in this affine coordinate chart, writing $f = f_1 + f_2 + \cdots + f_d$, where $f_i$ is a homogeneous polynomial of degree $i$.

Note that if we identify the space of lines in $\PP^n$ passing through $p$ with $\PP^{n-1}$, then the space of such lines that lie in $X$ will be the intersection of the $V(f_i).$ We therefore only need to analyze when this intersection has larger dimension than would be expected.

By Lemma \ref{kdimConds}, it is at least $\binom{d+k}{d}$ conditions for a hypersurface to contain a $k$-dimensional subvariety.  We consider separately the case where $X$ is singular at $p$ and $X$ is smooth at $p$.

First, suppose $X$ is singular at $p$, i.e., $f_1 = 0$.  Then $f_2$ will be non-zero outside of a codimension $\binom{n+1}{2}$ variety, and given that $f_2 \neq 0$, $V(f_3)$ will not contain any component of $V(f_2)$ outside of a codimension $\binom{n-2+3}{3} = \binom{n+1}{3}$ variety.  Similar, if $\bigcap_{1<i<j} V(f_i)$ has dimension $n-j+1$, $V(f_j)$ will not contain any component of $\bigcap_{1<i<j} V(f_i)$ outside of a $\binom{n-j+1+j}{j}=\binom{n+1}{j}$-codimensional variety.  For $n \geq d+2$, $\binom{n+1}{j} \geq \binom{n}{2}$ for $2 \leq j \leq d$.

Now suppose $X$ is nonsingular at $p$, i.e., $f_1 \neq 0$.  Then $f_2$ will not contain any component of $V(f_1)$ outside of a codimension $\binom{n-2+2}{2} = \binom{n}{2}$ variety.  Similarly, if $\bigcap_{i<j} V(f_i)$ has dimension n-j, $V(f_j)$ will not contain any component of $\bigcap_{i<j} V(f_i)$ outside of a codimension $\binom{n-j+j}{j} = \binom{n}{j}$ variety.  For $n \geq d+2$ and $2 \leq j \leq d$, we see that $\binom{n}{j} \geq \binom{n}{2}$.

\end{proof}

For the proof of our main result, we need to understand which hypersurfaces contain small-degree rational curves in their singular loci.  The following Proposition bounds the codimension of such hypersurfaces.

\begin{proposition}
\label{singHyps}
The space of degree $d$ hypersurfaces singular along a degree $e$ rational curve has codimension at least $n(d-e-1)-e+4$.
\end{proposition}
\begin{proof}
The space of degree $e$ rational curves in $\PP^n$ has dimension $(n+1)(e+1)-4$, so we just need to check that the space of hypersurfaces singular along a given degree $e$ rational curve $C$ has codimension at least $nd+1$, the codimension of the space of hypersurfaces singular along a line. We will reduce to this case by deforming $C$ to a line.

Without loss of generality, we may assume that $C$ does not intersect the $(n-2)$-plane $a_0=a_1=0.$ Now consider the closed subvariety $\mathcal{F}^\circ$ of $\mathbb{P}^n\times\mathbb{A}^1-\{0\}$ whose fiber above a point $r$ of $\mathbb{A}^1-\{0\}$ is the image of $C$ under the automorphism $[a_0:a_1:\cdots : a_n]\rightarrow [a_0:a_1:ra_2:ra_3:\cdots : ra_n]$ of $\mathbb{P}^n.$ Let $\mathcal{F}$ be the closure of $\mathcal{F}^\circ$ in $\mathbb{P}^n\times\mathbb{A}^1.$

The set theoretic fiber of $\mathcal{F}$ over $0$ is the line $a_2=a_3=\cdots=a_n=0$. We thus see that the dimension of the space of hypersurfaces singular everywhere along $C$ is at most the dimension of the space of hypersurfaces singular everywhere along a line, and we already worked out the dimension of the space of hypersurfaces singular along a line in the proof of Proposition \ref{1levelness}.  Thus, the codimension of the space of degree $d$ hypersurfaces singular along any degree $e$ rational curve is at most $nd+1 - ((n+1)(e+1)-4) = n(d-e-1)-e+4$

\end{proof}

For technical reasons in the proof of the main theorem we will need to show that an $e$-level hypersurface will contain lots of curves that aren't multiple covers of other curves, which is a result of independent interest.

\begin{proposition}
\label{genInj}
If $n \geq d+2$, $e \geq 2$ and $X$ is $(e-1)$-level, then in any component of the family of degree $e$ rational curves through $p$, there is a pair $(f,C) \in \overline{\MM}_{0,0}(X,e)$ such that $f$ is generically injective.
\end{proposition}
\begin{proof}

Let $k > 1$ be a factor of $e$.  We claim that the dimension of the space of degree $k$ covers of a degree $\frac{e}{k}$ curve is smaller than the dimension of curves through $p$.  We assume that $p$ is a smooth point of $X$, since the computation is similar if $p$ is a singular point (for $p$ singular, everything works the same except in the exceptional case $k=e=2$, we need to use the fact that there will be a $(n-d)$-family of lines through a singular point).  The space of degree $k$ covers of a degree $\frac{e}{k}$ curve through $p$ has dimension
\[ \frac{e}{k} \left( n-d+1 \right) - 2 + 2k-2 = \frac{e}{k} ( n-d+1 ) + 2k - 4 .\]
Any component of the family of degree $e$ rational curves through $p$ will have dimension at least
\[ e(n-d+1) - 2 .\]

Thus, we need only show that
\[ e(n-d+1) - 2 > \frac{e}{k} ( n-d+1 ) + 2k - 4 .\] 
Rearranging, we obtain
\[ e \left( 1-\frac{1}{k} \right)(n-d+1) > 2k-2 .\]
or
\[ e\left(n-d+1\right) > 2k.\]
which is clear, as $e \geq k$ and $n-d+1\geq 3$.

\end{proof}

\section{Proof of Main Result}
\label{mainResultSection}
The proof of our main result proceeds by inductively bounding the codimensions of the spaces of $e$-uneven hypersurfaces.  To do this, we show that if $\text{codim } S_{e-1} - \text{codim } S_e$ is too large, then we can find a large family of hypersurfaces and points with no reducible curves through the point.  We then apply Bend-and-Break to the family of curves in those hypersurfaces through those points.  We can imagine ``borrowing'' the curves from nearby hypersurfaces to have enough to apply Bend-and-Break.

We first need a technical lemma.

\begin{lemma}
\label{autLemma}
Let $\cC$ be the Chow variety of degree $e$ rational curves, and let $T \subset \cC$ be a $\operatorname{PGL}_{n+1}$-invariant family parameterizing at least one curve that is not a multiple line.  Then $\dim T \geq 3n-3$.
\end{lemma}
\begin{proof}
Let $\BB$ be the incidence correspondence 
\[ \{ (C,p_1,p_2,p_3) | p_i \in C, \: C \in T \} \subset \cC \times \PP^n\times\PP^n\times\PP^n.\]
Let $C \in T$ be a curve that is not a multiple line, and choose three points $p_1,p_2,p_3$ on $C$ which are not collinear. Then those three points can be sent to any other three non-collinear points in $\PP^n$ by an automorphism of $\PP^n.$ This shows that the dimension of the space $\cC_3 \subset \BB$ of triples $(C',p'_1,p'_2,p'_3)$ which can be obtained by applying an automorphism of $\PP^n$ to $(C,p_1,p_2,p_3)$ is at least $3n$. But the fibers of the projection $\BB\rightarrow\cC$ are $3$-dimensional, as $p_1,p_2,p_3$ must lie on $C$, so $T$ must have dimension at least $3n-3$.
\end{proof}

\begin{theorem}
\label{mainResult}
Denote by $M$ the codimension of $S_{e-1}$.  Then the codimension of $S_e$ in $\PP^N$ is at least $\min \{M, M - 2n + e(n-d+1)-1 , n(d-e-1)-e+4 \}$.
\end{theorem}
\begin{proof}
Note that $S_e$ is the union of three (possibly overlapping) sets: $S_{e-1}$, the space of hypersurfaces singular along a degree $e$ rational curve, and the closure of the space of hypersurfaces with an $e$-uneven point.  The codimension of $S_{e-1}$ is at least $M$ by assumption, and the codimension of the space of hypersurfaces singular along a degree $e$ rational curve is at least $n(d-e-1)-e+4$ by Proposition \ref{singHyps}.  Thus, it remains to bound the codimension of the space of hypersurfaces with an $e$-uneven point.  If $e(n-d+1)-1 \geq 2n$ (or indeed, if $e(n-d+1) \geq n+1$), then by Corollary \ref{lotsaCurvesCor} we see that any $(e-1)$-level hypersurface not singular along a degree $e$ rational curve will be $e$-level, so we need only consider the case $e(n-d+1)-1 < 2n$.  The statement is vacuous if $M \leq 2n-e(n-d+1)-1$, so we can assume $M > 2n-e(n-d+1)-1$.  We show that the closure of the space of hypersurfaces with an $e$-uneven point is codimension at least $M-2n+e(n-d+1)-1$, which will suffice to prove the theorem.

Suppose the result is false.  That is, suppose that the closure of the space of hypersurfaces with an $e$-uneven point has codimension at most $M-2n+e(n-d+1)-2$.  Denoting  
\[ \aA=\{ (p, f, C, X) \: | \: p \in f(C) \subset X, \text{the fiber of $\ev$ over $p$ contains a component } \]
 \[ \text{containing $(f,C)$ that is disjoint from the boundary $\Delta$} \} \] 
 \[\subset \overline{\MM}_{0,1}(\PP^n,e) \times \PP^N \]
we can find an irreducible component $\NN_e$ of the closure of $\aA$ such that the dimension of the projection of $\NN_e$ onto the space of hypersurfaces has codimension at most $M-2n+e(n-d+1)-2$.  Let $\cC$ be the Chow variety of rational degree $e$ curves in $\PP^n$, and let $\pi: \NN_e \to \cC$ be the natural map.  Let $\psi: \NN_e \to \PP^N$, $\phi: \NN_e \to \Phi$ and $\psi_1: \Phi \to \PP^N$ be given by the natural maps.  Note $\psi = \psi_1 \circ \phi$.

\begin{center}
$\begin{CD}
\FF  @>>>  \NN_e        @>>\pi> \cC\\
@.         @VV\phi V\\
@.         \Phi         @>>>    \PP^n\\
@.         @VV\psi_1V\\
@.         \PP^N\\
\end{CD}$
\end{center}

We claim that we can find a closed, irreducible family $\FF \subset \NN_e$ of dimension $2n-1$ with the following properties:
\begin{enumerate}
\item $\psi(\FF) \cap S_{e-1} = \emptyset$
\item If $(p,f,C,X) \in \FF$, then $C$ is irreducible.
\item $\textsf{dim} \: \pi(\FF) = 2n-1$
\end{enumerate}

First we prove the theorem assuming the claim.  Since $\pi(\FF)$ has dimension at least $2n-1$, by Corollary \ref{bendBreak} we see that $\FF$ must parameterize points $(p,f,C,X)$ with $C$ reducible, which contradicts property 2. (Condition 1 is needed to prove Condition 2).

Thus, it remains to prove the claim. We start by proving that $\pi(\NN_e)$ has dimension at least $3n-3$. By the definition of $\aA$, $\bar{\aA}$ is invariant under automorphisms of $\PP^n.$ We thus have a map $\operatorname{PGL}_{n+1}\times\NN_e \rightarrow \bar{\aA}.$ As $\operatorname{PGL}_{n+1}$ is irreducible, so is $\operatorname{PGL}_{n+1}\times\NN_e$, and thus the image of this map must be irreducible. But the image of this map contains $\NN_e$, so the image of this map must be $\NN_e.$ Thus, $\NN_e$ is preserved by automorphisms of $\PP^n$.  By Proposition \ref{genInj} there is a point $(p,f,C,X) \in \NN_e$ such that $f(C)$ is not a line, so by Lemma \ref{autLemma}, $\pi(\NN_e)$ has dimension at least $3n-3$.

%By Proposition \ref{genInj} there is a point $(p,f,C,X) \in \NN_e$ such that $f(C)$ is not a line.  Then, if we choose three points $p_1,p_2,p_3$ on $f(C)$ which are not collinear, then those three points can be sent to any other three non-collinear points in $\PP^n$ by an automorphism of $\PP^n.$ This shows that the dimension of the space $\cC_3 \subset \cC\times\PP^n\times\PP^n\times\PP^n$ of triples $(C',p'_1,p'_2,p'_3)$ which can be obtained by applying an automorphism of $\PP^n$ to $(C,p_1,p_2,p_3)$ is at least $3n$. But the fibers of the projection $\BB\rightarrow\cC$ are $3$-dimensional, as $p_1,p_2,p_3$ must lie on $C$, so the image of $\pi$ will have dimension at least $3n-3$.

We now construct $\FF$.  Let $c$ be the dimension of the generic fiber of $\phi: \NN_e \to \Phi$, and set $a = 2n-1 - c$.  If $c \geq 2n-1$, then choose $\FF$ to be a general $(2n-1)$-dimensional subvariety of a general nonempty fiber of $\phi$.  Checking the three conditions is slightly easier in this case.  For condition 1, $\psi(\FF)$ is simply a point in $\PP^N$, and a general such will be disjoint from $S_{e-1}$ by our assumptions at the beginning of this proof.  Condition 2 follows by the definition of $\NN_e$ and generality of $\FF$.  Condition 3 follows from Proposition \ref{genInj} and the fact that the curves of $\FF$ are all distinct elements of $\overline{\MM}_{0,n}(\PP^n,e)$.

Otherwise $c < 2n-1$, and the construction of $\FF$ essentially amounts to picking a subvariety of $\NN_e$ to contain as many different rational curves as possible.  In particular, we try to avoid selecting tuples $(X,C,f,p_t)$, where $X$, $C$, and $f$ are fixed while $p_t$ varies along $C$.   Let $H$ be a general plane in $\PP^N$ of dimension $N - \dim \psi(\NN_e)+a$, so that $H \cap \psi(\NN_e)$ has dimension $a$.  Let $\FF'$ be a general irreducible subvariety of $\psi_1^{-1}(H) \cap \phi(\NN_e)$ such that $\psi_1|_{\FF'}$ is a dominant, generically finite map onto $H \cap \psi(\NN_e)$.  Let $\FF$ be a component of $\phi^{-1}(\FF')$ that dominates $\FF'$.  By construction, the general fiber of the map $\psi|_{\FF}$ has dimension $c$, so that $\FF$ has dimension $a+c = 2n-1$.
 
We now check the three conditions.  We start with condition 1.  Because $c \geq e(n-d+1)-2$, $a \leq 2n-1 - (e(n-d+1) -2) = 2n-e(n-d+1)+1$.  By hypothesis $\dim \psi(\NN_e) - \dim S_{e-1} \geq 2n-e(n-d+1)+2 > a$, so it follows that $\psi(\FF) = \psi(\NN_e) \cap H$ is disjoint from $S_{e-1}$.

For condition 2, we see from generality of $\FF$ and definition of $\NN_e$ that for a general $(p,X) \in \phi(\FF)$, every $C$ with $(p,f,C,X) \in \FF$, will be irreducible.   Additionally, every hypersurface in $\psi(\FF)$ is $(e-1)$-level by condition 1.  Since the general fiber of $\phi|_{\FF}$ contains no irreducibles by construction, it follows from Corollary \ref{closedUneven} it follows that every $C$ with $(p,f,C,X) \in \FF$ will be irreducible.

To prove condition 3, we show that $\pi|_{\FF}$ is generically finite.  Let $(p,f,C,X) \in \FF$ be general.  We claim $\pi$ is finite at $(p,f,C,X)$.  By Proposition \ref{genInj} and generality of $(p,f,C,X)$, $f$ will be generically injective.  Define $\BB = \{(p',X') \in \Phi \: | \: p' \in f(C) \subset X' \} \cap \phi(\NN_e)$.  Since $(\pi, \phi)$ is injective by construction, $\BB \cap \phi(\FF)$ will be the fiber of $\pi$ over $f(C)$.  Since the image of $\pi$ had dimension at least $3n-3$, $\BB$ has codimension at least $3n-3$ in $\phi(\NN_e)$.  Since the fibers of $\psi_1|_{\BB}$ are $1$-dimensional while the fibers of $\psi_1$ are $(n-1)$-dimensional, this means $\psi_1(\BB)$ has codimension at least $2n-1$ in $\psi(\NN_e)$, which means that $H$ intersects $\psi_1(\BB)$ in at most finitely many points by generality of $H$.  Since $\psi_1|_{\FF'}$ is finite, this shows that $\pi$ is finite at $(p,f,C,X)$.  This suffices to show condition 3.

\end{proof}

The rest is just working out the numbers.  We know the result for $d \leq \frac{n+1}{2}$, so it remains to consider $d \geq \frac{n+1}{2}$.  If $n \leq 5$ then $d \leq n-2$ means $d \leq 3 = \frac{n+1}{2}$, so without loss of generality, we may assume $n \geq 6$.

\begin{corollary}
\label{dimSeCor}
If $d \geq \frac{n+1}{2}$ and $e \leq \frac{n+1}{n-d+1}$ then 
\[ \codim S_e \geq \binom{n}{2}+d-2en+\frac{e(e+1)}{2}(n-d+1) - e + 1 . \]
\end{corollary}
\begin{proof}
First, we show that $n(d-e-1)-e+4 \geq \binom{n}{2}+d-2en+\frac{e(e+1)}{2}(n-d+1)$.  We will first show that the inequality is strict for $e \leq \frac{n}{n-d+1}$.  For $e=1$ we have 
\[ n(d-2)+3 > \binom{n}{2} - n + 1 \]
is equivalent to
\[ (d-2) > \frac{1}{n} \binom{n}{2} - 1 - \frac{2}{n} \]
which, since $d \geq \frac{n+1}{2}$, is equivalent to
\[ d > \frac{1}{n} \binom{n}{2} + 1 - \frac{2}{n} = \frac{n-1}{2} + 1 - \frac{2}{n} = \frac{n+1}{2} - \frac{2}{n} . \]
For $e \leq \frac{n}{n-d+1}$, note that each time we replace $e-1$ with $e$, the left-hand side decreases by $n+1$, while the right hand side decreases by $2n-e(n-d+1)+1$.  We see that $2n-e(n-d+1)+1 \geq n+1$, which together with the base case $e=1$ shows that $n(d-e-1)-e+4 > \binom{n}{2}+d-2en+\frac{e(e+1)}{2}(n-d+1)$ for this range of $e$.  For $e = \frac{n+1}{n-d+1}$, we see that replacing $e-1$ with $e$ decreases the right-hand side by at least $n$.  Together with the fact that the inequality was strict for $e-1$, this proves $n(d-e-1)-e+4 \geq \binom{n}{2}+d-2en+\frac{e(e+1)}{2}(n-d+1)$ for $e \leq \frac{n+1}{n-d+1}$.

Now we prove the entire statement of the corollary by induction.  For the base case $e=1$, we need to show $\codim S_1 \geq \binom{n}{2} + d - 2n + (n-d+1) = \binom{n}{2} - n + 1$, which follows from Proposition \ref{1levelness} and the above discussion.

Finally, we proceed with the induction step.  Suppose $\codim S_{e-1} \geq \binom{n}{2} + d - 2(e-1)n + \frac{e(e-1)}{2}(n-d+1)-e+2$.  By Theorem \ref{mainResult}, we see that $\codim S_e \geq \min \{ \codim S_{e-1}, \codim S_{e-1}-2n+e(n-d+1) - 1, n(d-e-1) - e +4 \}$.  Using the induction hypothesis and the fact that $n(d-e-1)-e+4 \geq \binom{n}{2}+d-2en+\frac{e(e+1)}{2}(n-d+1)$, we see that
\[ \codim S_e \geq \binom{n}{2} + d - 2en + \frac{e(e+1)}{2}(n-d+1)  - e + 1. \]
\end{proof}

\begin{corollary}
If $X$ is a general hypersurface of degree $d$ in $\PP^n$ and $n \geq d+2$, then the space of rational, degree $e$ curves through an arbitrary point $p \in X$ has the expected dimension for all $e$.
\end{corollary}
\begin{proof}
By Corollary \ref{lotsaCurvesCor} the threshold degree is $\lfloor \frac{n+1}{n-d+1} \rfloor$, so by Corollary \ref{dimSeCor}, it remains to show that $\binom{n}{2}+d-2en+\frac{e(e+1)}{2}(n-d+1)-e+1$ is positive for $1 \leq e \leq \frac{n+1}{n-d+1}$.  Multiplying by two, it suffices to show
\[ n(n-1) - 4en + e(e+1)(n-d+1) + 2d -2e +2 > 0 .\]
The expression on the left is decreasing in $e$ for $e \leq \frac{n+1}{n-d+1}$, so it suffices to prove the result for $e = \frac{n+1}{n-d+1}$.  Dividing by $e$ gives
\[ \frac{n(n-1)}{n+1}(n-d+1) - 4n + (e+1)(n-d+1) + \frac{2d}{e} - 2 + \frac{2}{e} > 0 \]
which can be rearranged to obtain
\[ \frac{n^2+1}{n+1}(n-d+1) - 4n + n+1 + \frac{2d}{e}-2 + \frac{2}{e} > 0 \]
or
\[ \frac{n^2+1}{n+1}(n-d+1) + \frac{2d}{e} + \frac{2}{e} > 3n+1 .\]
Multiplying both sides by $n+1$, we get
\[ (n^2+1)(n-d+1) + 2d(n-d+1)+2(n-d+1) > 3n^2+4n+1 , \]
or 
\[ (n^2+2d+3)(n-d+1) > 3n^2+4n+1 .\]
The left-hand side is quadratic in $d$ with negative coefficient of $d^2$, so we need only check the endpoints to minimize it.  If $d=n-2$, the left-hand side becomes
\[ 3(n^2+2n-1) = 3n^2+6n-3. \]
This will be greater than $3n^2+4n+1$ precisely when $2n > 4$, or $n > 2$.

If $d = \frac{n+1}{2}$, the left-hand side is 
\[ (n^2+n+4)\frac{n+1}{2} .\]
For $n \geq 5$, we have
\[ (n^2+n+4)\frac{n+1}{2} = (n^2+n)\frac{n+1}{2} + 2(n+1) \geq 3n^2+3n+2n+2 > 3n^2+4n+1 .\]

This concludes the proof.
%For $n \geq d+2$ and $d \geq \frac{n+1}{2}$, the left-hand side of the last inequality is at least
%\[ 3(n^2+1)+3(n+1) = 3n^2+3n+6 > 3n^2+2n-1 .\]
%The result follows.
\end{proof}

\section{Conclusion}
This paper covers Conjecture \ref{expDim} for much of the Fano range.  However, as we point out in the introduction, there remain many ranges of $n$, $d$, and $e$ for which we do not know whether Conjecture \ref{expDim} is true.  This includes a few more cases in the Fano range ($n = d+1$ and $n=d$), a large swath of cases in the general type range, and the Calabi-Yau range, including the Clemens Conjecture.  We hope that more progress will be made in the future.

\bibliographystyle{plain}
\bibliography{breakingandborrowing}
\end{document}